\documentclass[10pt,a4paper]{article}
\usepackage[latin1]{inputenc}
\usepackage{amsmath}
\usepackage{amsfonts}
\usepackage{amssymb}
\usepackage{amsthm}
\usepackage{makeidx}
\usepackage{natbib}
\bibpunct{[}{]}{;}{n}{,}{,}

\usepackage[margin=3cm]{geometry}
\allowdisplaybreaks

\author{Mirko D'Ovidio\footnote{{\em Department of Statistics, Probability and Applied Statistics. Sapienza University of Rome, P.le Aldo Moro, 5 - 00185 Rome (Italy). e-mail: mirko.dovidio@uniroma1.it}}}

\title{Explicit solutions to fractional diffusion equations via Generalized Gamma Convolution}

\newtheorem{te}{Theorem}

\newtheorem{lm}{Lemma}
\newtheorem{os}{Remark}
\newtheorem{coro}{Corollary}

\renewcommand{\H}{H^{m,n}_{p,q}\left[ x \bigg| \begin{array}{l} (a_i, \alpha_i)_{i=1, .. , p}\\ (b_j, \beta_j)_{j=1, .. , q}  \end{array} \right]}

\numberwithin{equation}{section}

\begin{document}

\maketitle

\textbf{Abstract} In this paper we deal with Mellin convolution of generalized Gamma densities which leads to integrals of modified Bessel functions of the second kind. Such convolutions allow us to explicitly write the solutions of the time-fractional diffusion equations involving the adjoint operators of a square Bessel process and a Bessel process.
\newline 

\textbf{Keywords}: Mellin convolution formula, generalized Gamma r.v.'s, Stable subordinators, Fox functions, Bessel processes, Modified Bessel functions.

\section{Introduction and main result}
In the last years, the analysis of the compositions of processes and the corresponding governing equations has received the attention of many researchers. Many of them are interested in compositions involving subordinators, in other words, subordinated processes $Y(T(t))$, $t>0$ (according to \cite{Fel71}) where $T(t)$, $t>0$ is a random time with non-negative, independent and homogeneous increments (see \cite{Btoi96}). If the random time is a (symmetric or totally skewed) stable process we have results which are strictly related to the Bochner's subordination and the p.d.e.'s connections have been investigated, e.g., in \cite{DEBL04, DO, DO2, BMN09ann, NA08}. If the random time is an inverse stable subordinator we shall refer to the governing equation of $Y(T(t))$ as a fractional equation considering that a fractional time-derivative must be taken into account. In the literature, several authors have studied the solutions to space-time fractional equations. In the papers by Wyss \cite{Wyss86}, Schneider and Wyss \cite{SWyss89}, the authors present solutions of the fractional diffusion equation $\partial_t^{\lambda} T = \partial^2_x T$ in terms of Fox's functions (see Section \ref{prelim}). In the works by Mainardi et al., see e.g. \cite{MLP01, MPG03} the authors have shown that the solutions to space-time fractional equation $_xD^{\alpha}_{\theta} u = \, _tD^{\beta}_{*} u$ can be represented by means of Mellin-Barnes integral representations (or Fox's functions) and M-Wright functions (see e.g. \citet{KST06}). The fractional Cauchy problem $D^{\alpha}_t u = L\, u$ has been thoroughly studied by yet other authors and several representations of the solutions have been carried out, but an explicit form of the solutions has never been obtained. \citet{Nig86} gave a physical derivation when $L$ is the generator of some continuous  Markov process. \citet{Zasl94} introduced the space-time fractional kinetic equation for Hamiltonian chaos. \citet{Koc89, Koc90} first introduced a mathematical approach while \citet{BM01} established the connections between fractional problem and subordination by means of inverse stable subordinator when $L$ is an infinitely divisible generator on a finite dimensional vector space. In particular, if $\partial_t p = Lp$ is the governing equation of $\mathbf{X}(t)$, then under certain conditions, $\partial^\beta_t q = Lq + \delta(\mathbf{x}) t^{-\beta}/\Gamma(1-\beta)$ is the equation governing the process $\mathbf{X}(V_t)$ where $V_t$ is the inverse or hitting time process to the $\beta$-stable subordinator, $\beta \in (0,1)$. \citet{OB04, OB09} found explicit representations of the solutions to $\partial^{\nu}_t u = \lambda^2 \partial^2_x u$ only in some particlular cases: $\nu=(1/2)^{n}$, $n \in \mathbb{N}$ and $\nu=1/3,2/3,4/3$. Also, they represented the solutions to the fractional telegraph equations in terms of stable densities, see \cite{OB03, OB04}. In general, the solutions to fractional equations represent the probability densities of certain subordinated processes obtained by using a time clock (in the following we will refer to it as  $L^{\nu}_t$) which is an inverse stable subordinator (see Section \ref{secSSub}). For a short review on this field, see also \citet{NANERW} and the references therein.\\

We will present the role of the Mellin convolution formula in finding solutions of fractional diffusion equations. In particular, our result allows us to write the distribution of both stable subordinator and its inverse process whose governing equations are respectively space-fractional or time-fractional equations. This result turns out to be useful for representing the solutions to the following fractional diffusion equation
\begin{equation}
D^{\nu}_t \,\tilde{u}^{\gamma, \mu}_{\nu} = \mathcal{G}_{\gamma, \mu}\, \tilde{u}^{\gamma, \mu}_{\nu} \label{ProbMain}
\end{equation}
where $\tilde{u}^{\gamma, \mu}_{\nu}=\tilde{u}^{\gamma, \mu}_{\nu}(x,t)$, $x>0$, $t>0$, $D^{\nu}_t$ is the Riemann-Liouville fractional derivative, $\nu \in (0,1]$ and $\mathcal{G}_{\gamma, \mu}$ is an operator to be defined below (see formula \eqref{generatorG}). We present,  for $\nu = 1/(2n+1)$, $n \in \mathbb{N} \cup \{0\}$, the explicit solutions to \eqref{ProbMain} in terms of integrals of modified Bessel functions of the second kind ($K_\nu$) whereas, for $\nu \in (0,1]$, we obtain the solutions to \eqref{ProbMain} in terms of Fox's functions. After some preliminaries in Section \ref{prelim}, in Section \ref{secConv} we recall the generalized Gamma density $Q^{\gamma}_{\mu}$ starting from which we define the distribution $g^{\gamma}_{\mu}$ of the (generalized Gamma) process $G^{\gamma, \mu}_t$ and the distribution $e^{\gamma}_{\mu}$ of the process $E^{\gamma, \mu}_t$. The latter can be seen as the reciprocal Gamma process, indeed $E^{\gamma, \mu}_t = 1/G^{\gamma, \mu}_t$, or in a more striking interpretation, as the hitting time process for which $(E^{\gamma, \mu}_t < x) = (G^{\gamma, \mu}_x > t)$. We shall refer to $E^{\gamma, \mu}_t$ as the reciprocal or equivalently the inverse process of $G^{\gamma, \mu}_t$. It must be noticed that $e^{\gamma}_{\mu} = g^{-\gamma}_{\mu}$ because  $G^{-\gamma, \mu}_t = 1/G^{\gamma, \mu}_t$. Furthermore, we introduce the most important tool we deal with in this paper, the Mellin convolutions $g^{\gamma, \star n}_{\bar{\mu}}$ (see formula \eqref{mcgW}) and $e^{\star n}_{\bar{\mu}}$ (see formula \eqref{mceW}) where $e^{\star n}_{\bar{\mu}}$ stands for $e^{1,\star n}_{\bar{\mu}}$.  In Section \ref{secSSub} we draw some useful transforms of the distribution $h_{\nu}$ of the stable subordinator $\tilde{\tau}^{\nu}_t$ and the distribution $l_{\nu}$ of the inverse process $L^{\nu}_t$. Similar calculations can be found in the paper by \citet{SWyss89}. The inverse (or hitting time) process is defined once again from the fact that $(L^{\nu}_t <x) = (\tilde{\tau}^{\nu}_{x} > t)$ (see also \cite{BM01, Btoi96}). In Section \ref{mainRes} we present our main contribution. We show that the following representations hold true:
\begin{equation*}
h_{\nu}(x,t) = e^{\star n}_{\bar{\mu}}(x, \varphi_{n+1}(t)), \quad x>0,\; t>0,\; \nu =1/(n+1), \; n \in \mathbb{N}
\end{equation*}
and
\begin{equation*}
l_{\nu}(x,t) = g^{(n+1), \star n}_{\bar{\mu}}(x, \psi_{n+1}(t)), \quad x>0, \; t>0, \; \nu=1/(n+1),\; n \in \mathbb{N}.
\end{equation*}
where $\bar{\mu}=(\mu_1, \ldots , \mu_n)$, $\mu_j=j\,\nu$, $j=1,2,\ldots, n$, $\nu=1/(n+1)$, $n \in \mathbb{N}$ and the time-stetching functions are given by $\varphi_m(s) = (s/m)^m$ and $\psi_m(s)=ms^{1/m}$, $s \in (0, \infty)$, $m \in \mathbb{N}$, $\psi=\varphi^{-1}$.\\  
The discussion made so far allows us to introduce the result stated in Theorem \ref{mainTheorem}. For $\nu = 1/(n+1)$, $n \in \mathbb{N} \cup \{0\}$, the solutions to \eqref{ProbMain} can be written as follows
\begin{equation*}
\tilde{u}^{\gamma, \mu}_{\nu}(x, t) = \int_0^\infty g^\gamma_\mu (x,s^{1/\gamma}) \; g^{1/\nu, \star (1 /\nu - 1)}_{\bar{\mu}}(s, \psi_{1/\nu}(t)) \; ds, \quad x \in (0, \infty),\; t>0
\end{equation*}
where, for $n \in 2\mathbb{N} \cup \{0\}$, we have
\begin{equation*}
g^{1/\nu, \star (1 /\nu - 1)}_{\bar{\mu}}(x, t) = \frac{1}{\nu^{1/2\nu}} \left( \frac{x}{\pi^2 t^3} \right)^\frac{1-\nu}{4\nu} \int_0^\infty \ldots \int_0^\infty \mathcal{Q}_{\frac{1-\nu}{2}}(x, s_1) \ldots \mathcal{Q}_{\frac{1-\nu}{2}}(s_{n-1}, t) ds_1 \ldots ds_{n-1}
\end{equation*}
and
\begin{equation*}
\mathcal{Q}_{\frac{1-\nu}{2}}(x,t)=K_{\frac{1-\nu}{2}} \left( 2\sqrt{ (x / t)^{1/\nu}} \right), \quad x>0,\; t>0.
\end{equation*}
As a direct consequence of this result we obtain $\tilde{u}^{\gamma, \mu}_{1} = \tilde{g}^{\gamma}_{\mu}$, for $\nu =1$,  where $\tilde{g}^{\gamma}_{\mu}(x,t) = g^{\gamma}_{\mu}(x, t^{1/\gamma})$ and the governing equation writes
\begin{equation*}
\frac{\partial}{\partial t} \tilde{u}^{\gamma, \mu}_{1} = \frac{1}{\gamma^2} \left( \frac{\partial}{\partial x} x^{2-\gamma} \frac{\partial}{\partial x} - (\gamma \mu -1) \frac{\partial}{\partial x} x^{1-\gamma} \right) \tilde{u}^{\gamma, \mu}_{1}, \quad x > 0, \; t>0.
\end{equation*}
Furthermore, for $\gamma=1,2$ and $\nu \in (0,1]$ we obtain
\begin{equation}
D^{\nu}_t \tilde{u}^{1,\mu}_{\nu} = \left( x \frac{\partial^2}{\partial x^2} - (\mu -2) \frac{\partial}{\partial x} \right) \tilde{u}^{1,\mu}_{\nu}, \quad x > 0,\; t>0, \; \mu>0
\label{p1}
\end{equation}
and
\begin{equation}
D^{\nu}_t  \tilde{u}^{2,\mu}_{\nu} = \frac{1}{2^2} \left( \frac{\partial^2}{\partial x^2} - \frac{\partial}{\partial x} \frac{(2\mu -1)}{x} \right) \tilde{u}^{2,\mu}_{\nu}, \quad x > 0,\; t>0, \; \mu>0.
\label{p2}
\end{equation}
Equation \eqref{p2} represents a fractional diffusion around spherical objects and thus, the solutions we deal with obey radial diffusion equations. 

\section{Preliminaries}
\label{prelim}
The H functions were introduced by Fox \cite{FOX61} in 1996 as a very general class of functions. For our purpose, the Fox's H functions will be introduced as the class of functions uniquely identified by their Mellin transforms. A function $f$ for which the following Mellin transform exists
\begin{equation*}
\mathcal{M}[ f(\cdot) ] (\eta) = \int_0^\infty x^{\eta} f(x) \frac{dx}{x}, \quad \Re \{ \eta \} > 0
\end{equation*}
can be written in terms of H functions by observing that
\begin{equation}
\int_{0}^{\infty} x^{\eta } H^{m,n}_{p,q}\left[ x \bigg| \begin{array}{l} (a_i, \alpha_i)_{i=1, .. , p}\\ (b_j, \beta_j)_{j=1, .. , q}  \end{array} \right] \frac{dx}{x} = \mathcal{M}^{m,n}_{p,q} (\eta), \quad \Re \{ \eta \} \in \mathcal{D}
\label{intStrip}
\end{equation}
where
\begin{equation}
\mathcal{M}^{m,n}_{p,q} (\eta)  = \frac{\prod_{j=1}^{m} \Gamma(b_j + \eta \beta_j) \prod_{i=1}^{n} \Gamma(1-a_i - \eta \alpha_i)}{\prod_{j=m+1}^{q} \Gamma(1-b_j - \eta \beta_j) \prod_{i=n+1}^p \Gamma(a_i + \eta \alpha_i)}.
\label{mellinHfox}
\end{equation}
The inverse Mellin transform is defined as
\begin{equation*}
f(x)=\frac{1}{2\pi i}\int_{ \theta -i \infty}^{\theta + i \infty} \mathcal{M}[ f(\cdot) ] (\eta) x^{-\eta} d\eta 
\end{equation*}
at all points $x$ where $f$ is continuous and for some real $\theta$. Thus, according to a standard notation, the Fox H function is defined as follows
\begin{equation*}
\H = \frac{1}{2\pi i}\int_{\mathbb{P}(\mathcal{D})} \mathcal{M}^{m,n}_{p,q}(\eta) x^{-\eta} d\eta
\end{equation*}
where $\mathbb{P}(\mathcal{D})$ is a suitable path in the complex plane $\mathbb{C}$ depending on the fundamental strip ($\mathcal{D}$) such that the integral \eqref{intStrip} converges. For an extensive discussion on this function see \citet{FOX61, MS73}. The Mellin convolution formula
\begin{equation} 
f_1 \star f_2 (x) = \int_0^\infty f_1(x/s) f_2(s) \frac{ds}{s}, \quad x>0
\label{mcFormula}
\end{equation}
turns out to be very useful later on. Formula \eqref{mcFormula} is a convolution in the sense that
\begin{equation} 
\mathcal{M} \left[ f_1 \star f_2(\cdot) \right](\eta) = \mathcal{M}\left[ f_1(\cdot)\right](\eta) \times \mathcal{M} \left[f_2(\cdot)\right](\eta) .
\label{MconvFormula}
\end{equation}
Throughout the paper we will consider the integral
\begin{equation}
f_1 \circ f_2(x,t) = \int_0^\infty f_1(x,s)f_2(s,t) ds 
\label{circConv}
\end{equation}
(for some well-defined $f_1$, $f_2$) which is not, in general, a Mellin convolution. We recall the following connections between Mellin transform and both integer and fractional order derivatives. In particular, we consider a rapidly decreasing function $f:[0, \infty) \mapsto [0,\infty)$, if there exists $a \in \mathbb{R}$ such that 
\begin{equation*}
\lim_{x \to 0^+} x^{a - k - 1} \frac{d^k}{d x^k} f(x) = 0, \quad k=0,1,\ldots , n-1, \quad n \in \mathbb{N}, \; x \in \mathbb{R}_+ 
\end{equation*}
then we have
\begin{align}
\mathcal{M} \left[ \frac{d^n}{d x^n}f(\cdot) \right] (\eta) = & (-1)^n \frac{\Gamma(\eta)}{\Gamma(\eta - n)} \mathcal{M} \left[ f(\cdot) \right] (\eta - n) \label{derMint}
\end{align}
and, for $0 < \alpha <1$
\begin{align}
\mathcal{M} \left[ \frac{d^\alpha}{d x^\alpha}f(\cdot) \right] (\eta) = & \frac{\Gamma(\eta)}{\Gamma(\eta - \alpha)}\mathcal{M}\left[ f(\cdot) \right](\eta -\alpha) \label{derMfrac}
\end{align}
(see \citet{KST06, SKM93} for details). The fractional derivative appearing in \eqref{derMfrac} must be understood as follows
\begin{equation}
\frac{d^\alpha}{d x^\alpha}f(x) = \frac{1}{\Gamma\left( n-\alpha \right)} \int_0^x (x-s)^{n-\alpha -1} \frac{d^n f}{d s^n}(s) \, ds, \quad n-1 < \alpha < n \label{Cfracder}
\end{equation}
that is the Dzerbayshan-Caputo sense. We also deal with the Riemann-Liouville fractional derivative
\begin{equation}
D^{\alpha}_x f = \frac{1}{\Gamma\left( n-\alpha \right)} \frac{d^n}{d x^n} \int_0^x (x-s)^{n-\alpha -1} f(s) \, ds, \quad n-1 < \alpha < n \label{Rfracder}
\end{equation}
and the fact that 
\begin{equation}
D^\alpha_x f = \frac{d^\alpha}{d x^\alpha} f - \sum_{k=0}^{n-1} \frac{d^k}{d x^k} f \Bigg|_{x=0^+} \frac{x^{k - \alpha}}{\Gamma(k - \alpha +1)}, \quad n-1 < \alpha < n, \label{RCfracder}
\end{equation}
see \citet{GM97} and \citet{KST06}. We refer to \citet{KST06, SKM93} for a close examination of the fractional derivatives \eqref{Cfracder} and \eqref{Rfracder}.

\section{Mellin convolution of generalized Gamma densities}
\label{secConv}
In this section we introduce and study the Mellin convolution of generalized gamma densities. In the literature, it is well-known that generalized Gamma r.v. possess density law given by
\begin{equation*}
Q^{\gamma}_{\mu}(z) = \gamma \frac{z^{\gamma \mu -1}}{\Gamma\left( \mu \right)} \exp \left\lbrace -z^\gamma \right\rbrace , \quad z>0,\, \gamma>0,\, \mu>0.
\end{equation*}
Our discussion here concerns the function 
\begin{equation}
g^{\gamma}_{\mu}(x,t) = \textrm{sign}(\gamma)\frac{1}{t} Q^{\gamma }_{\mu}\left(\frac{x}{t} \right) = | \gamma | \frac{x^{\gamma \mu -1} }{t^{\gamma \mu} \Gamma(\mu)} \exp\left\lbrace - \frac{x^\gamma}{t^\gamma} \right\rbrace , \qquad x > 0, \, t>0, \; \gamma \neq 0, \; \mu >0.
\label{dist:QQgenGamma}
\end{equation}
Let us introduce the convolution
\begin{equation}
g^{\gamma_1}_{\mu_1} \star g^{\gamma_2}_{\mu_2}(x,t) = \int_0^\infty g^{\gamma_1}_{\mu_1}(x,s) g^{ \gamma_2}_{\mu_2} (s,t) ds = \textrm{sign}(\gamma_1 \gamma_2) \frac{1}{t}\int_0^\infty Q^{\gamma_1 }_{\mu_1}(x/s) Q^{\gamma_2 }_{\mu_2}(s/t) \frac{ds}{s}
\label{melConv2}
\end{equation}
for which we have (see formula \eqref{MconvFormula})
\begin{equation}
\mathcal{M}\left[ g^{\gamma_1}_{\mu_1} \star g^{\gamma_2}_{\mu_2}(\cdot ,t) \right] (\eta) =  \mathcal{M}\left[ g^{\gamma_1}_{\mu_1}(\cdot, t^{1/2}) \right] (\eta) \times \mathcal{M}\left[ g^{\gamma_2}_{\mu_2}(\cdot ,t^{1/2}) \right] (\eta)
\label{me2con}
\end{equation}
as a straightforward calculation shows. We now introduce the generalized Gamma process (GGP in short). Roughly speaking, the function \eqref{dist:QQgenGamma} can be viewed as the distribution of a GGP $\{ G^{\gamma, \mu}_t, \, t>0 \}$ in the sense that $\forall t$ the distribution of the r.v. $G^{\gamma, \mu}_t$ is the generalized Gamma distribution \eqref{dist:QQgenGamma}. Thus, we make some abuse of language by considering a process without its covariance structure. In the literature there are several non-equivalent definitions of the distribution on $\mathbb{R}^n_+$ of Gamma distributions, see e.g. \citet{KoBa06} for a comprehensive  discussion. In Section \ref{mainRes} (Corollary \ref{coro1}) we will show that the distribution \eqref{dist:QQgenGamma} satisfies the p.d.e.
\begin{equation*}
\frac{\partial}{\partial t} \, g^{\gamma}_{\mu} = \frac{d (t^{\gamma})}{d t} \, \mathcal{G}_{\gamma, \mu}\, g^{\gamma}_{\mu}, \qquad x > 0,\, t>0
\end{equation*}
where 
\begin{equation}
\mathcal{G}_{\gamma, \mu} \, f= \frac{1}{\gamma^2} \left( \frac{\partial}{\partial x} x^{2-\gamma} \frac{\partial}{\partial x}  - (\gamma \mu -1) \frac{\partial}{\partial x} x^{1-\gamma} \right) f, \quad x > 0, \; t>0 
\label{generatorG}
\end{equation}
and $\gamma \neq 0$, $f \in D(\mathcal{G}_{\gamma, \mu})$. For $\gamma=1$,  equation \eqref{dist:QQgenGamma} becomes the distribution of a $2\mu$-dimensional squared Bessel process $\{BESSQ_{t/2}^{(2\mu)},\, t>0 \}$ and, for $\gamma=2$ we obtain the distribution of a $2\mu$-dimensional Bessel process $\{BES_{t/2}^{(2\mu)},\, t>0 \}$, both  starting from zero. Some interesting distributions can be realized through Mellin convolution of distribution $g^\gamma_\mu$. Indeed, after some algebra we arrive at
\begin{equation}
g^{\gamma}_{\mu_1} \star g^{- \gamma}_{\mu_2}(x,t) = \frac{\gamma }{B(\mu_1, \mu_2)} \frac{x^{\gamma \mu_1 -1} t^{\gamma \mu_2}}{(t^\gamma + x^\gamma)^{\mu_1 + \mu_2}}, \quad x >0, \, t>0, \, \gamma >0  \label{dist:wPP}
\end{equation}
and
\begin{equation}
g^{-\gamma}_{\mu_1} \star g^{\gamma}_{\mu_2}(x,t) = \frac{\gamma }{B(\mu_1, \mu_2)} \frac{x^{\gamma \mu_2 -1} t^{\gamma \mu_1}}{(t^\gamma + x^\gamma)^{\mu_1 + \mu_2}}, \quad x >0, \, t>0, \, \gamma >0 
\label{dist:wPPP}
\end{equation}
where $B(\cdot, \cdot)$ is the Beta function (see e.g. \citet[formula 8.384]{GR}). Moreover, in light of the Mellin convolution formula \eqref{MconvFormula}, the following holds true
\begin{equation*} 
\mathcal{M}\left[g^{\gamma}_{\mu_1} \star g^{- \gamma}_{\mu_2}(\cdot ,t) \right](\eta) = \mathcal{M} \left[ g^{-\gamma}_{\mu_2} \star g^{\gamma}_{\mu_1}(\cdot ,t) \right] (\eta). 
\end{equation*}
A further distribution arising from convolution can be presented. In particular, for $\gamma \neq 0$, we have
\begin{equation}
g^{\gamma}_{\mu_1} \star g^{\gamma}_{\mu_2} (x,t) = \frac{2| \gamma | \, \left(x^\gamma /t^\gamma \right)^\frac{\mu_1+\mu_2}{2}}{x\, \Gamma(\mu_1)\Gamma(\mu_2)} K_{\mu_2-\mu_1} \left( 2 \sqrt{\frac{x^\gamma}{t^\gamma}} \right), \quad x > 0, \; t>0  \label{eConv}
\end{equation}
which proves to be very useful further on. The function $K_{\nu}$ appearing in \eqref{eConv} is the modified Bessel function of imaginary argument (see e.g \cite[formula 8.432]{GR}). For the sake of completeness we have writen the following Mellin transforms:
\begin{equation*}
\mathcal{M}\left[ g^{\gamma}_{ \mu}(\cdot, t) \right] (\eta)= \frac{\Gamma\left(\frac{\eta -1}{\gamma} + \mu \right)}{\Gamma\left( \mu \right)}t^{\eta -1}, \quad t>0, \; \Re\{ \eta \} > 1-\gamma \mu, \; \gamma \neq 0,
\end{equation*}
and
\begin{equation}
\mathcal{M}\left[ g^{\gamma}_{ \mu}(x, \cdot) \right] (\eta) = \frac{\Gamma\left( \mu -\frac{\eta}{\gamma} \right)}{\Gamma\left(\mu\right)} x^{\eta -1}, \quad x>0, \; \Re\{ \eta \} > \gamma \mu, \; \gamma \neq 0.\label{lastM} 
\end{equation}
Formula \eqref{lastM} suggests that
\begin{equation*}
\mathcal{M}\left[ g^{\gamma_1}_{\mu_1} \star g^{\gamma_2}_{\mu_2}(x, \cdot) \right] (\eta) =  \mathcal{M}\left[ g^{\gamma_1}_{\mu_1}(x^{1/2}, \cdot) \right] (\eta) \times \mathcal{M}\left[ g^{\gamma_2}_{\mu_2}(x^{1/2}, \cdot) \right] (\eta).
\end{equation*}
For the one-dimensional GGP we are able to define the inverse generalized Gamma process $\{E_t^{\gamma, \mu}$, $t>0\}$ (IGGP in short) by means of the following relation
\[ Pr \{ E_t^{\gamma, \mu} <x \} = Pr \{ G^{\gamma, \mu}_x >t  \}. \]
The density law $e^{\gamma}_{\mu}=e^{\gamma}_{\mu}(x,t)$ of the IGGP can be carried out by observing that
\begin{equation}
e^{\gamma}_{\mu}(x,t) = Pr \{ E^{\gamma, \mu}_t \in dx \} / dx = \int_t^\infty \frac{\partial}{\partial x} g^{\gamma}_{\mu}(s, x) \, ds, \quad x>0,\, t>0 \label{derintgsign}
\end{equation}
and, making use of the Mellin transform, we obtain
\begin{align*}
\mathcal{M}\left[ e^{\gamma}_{\mu}(\cdot, t) \right](\eta) = & \int_t^\infty \mathcal{M} \left[ \frac{\partial}{\partial x} g^{\gamma}_{\mu}(s, \cdot) \right](\eta) \, ds, \quad \Re\{ \eta \} <1\\
= & \left[ \textrm{by } \eqref{derMint} \right] = - (\eta -1) \int_t^\infty \mathcal{M} \left[ g^{\gamma}_{\mu}(s, \cdot) \right](\eta -1) \, ds\\
= & \left[ \textrm{by } \eqref{lastM} \right] =  - (\eta -1) \int_t^\infty  \frac{\Gamma\left( \mu -\frac{\eta -1}{\gamma} \right)}{\Gamma\left(\mu\right)} s^{\eta -2}\, ds = \frac{\Gamma\left( \mu -\frac{\eta - 1}{\gamma} \right)}{\Gamma\left(\mu\right)} t^{\eta -1}
\end{align*}
The derivative under the integral sign in \eqref{derintgsign} is allowed from the fact that $\Xi_1(s)=\frac{\partial}{\partial x}g^{\gamma}_{\mu}(s,x) \in L^{1}(\mathbb{R}_{+})$ as a function of $s$. From \eqref{mellinHfox} and the fact that 
\begin{equation}
\H = c\; H^{m,n}_{p,q}\left[ x^c \bigg| \begin{array}{l} (a_i, c \alpha_i)_{i=1, .. , p}\\ (b_j, c \beta_j)_{j=1, .. , q}  \end{array} \right] 
\label{propH1}
\end{equation}
for all $c>0$ (see \citet{MS73}), we have that
\begin{equation}
e^{\gamma}_{\mu}(x, t) = \frac{\gamma}{x} H^{1,0}_{1,1}\left[ \frac{t^{\gamma}}{x^\gamma} \Bigg| \begin{array}{c} (\mu, 0)\\ (\mu, 1) \end{array} \right], \quad x>0,\, t>0,\, \gamma >0.  \label{distEH}
\end{equation}
By observing that $\mathcal{M}\left[ e^{\gamma}_{\mu}(\cdot, t)\right](1)=1$, we immediately verify that \eqref{distEH} integrates to unity. The density law $g^\gamma_\mu$ can be expressed in terms of H functions as well, therefore we have
\begin{equation}
g^{\gamma}_{\mu}(x, t) = \frac{\gamma}{x} H^{1,0}_{1,1}\left[ \frac{x^{\gamma}}{t^\gamma} \Bigg| \begin{array}{c} (\mu, 0)\\ (\mu, 1) \end{array} \right], \quad x>0,\, t>0,\, \gamma >0.
\label{distGH}
\end{equation}
In view of \eqref{distEH} and \eqref{distGH} we can argue that
\begin{equation*} 
E^{\gamma, \mu}_t \stackrel{law}{=} G^{-\gamma, \mu}_t \stackrel{law}{=} 1 / G^{\gamma, \mu}_t ,  \quad t>0, \, \gamma >0,\, \mu >0 
\end{equation*}
and $e^{\gamma}_{\mu}(x,t) = g^{-\gamma}_\mu(x,t)$, $\gamma>0$, $x>0$, $t>0$.

\begin{os}
\normalfont
We notice that the inverse process $\{ E^{1,1/2}_t, \, t>0 \}$ can be written as
\[ E^{1,1/2}_t = \inf\{s; \, B(s) = \sqrt{2t} \} \]
where $B$ is a standard Brownian motion. Thus, $E^{1,1/2}$ can be interpreted as the first-passage time of a standard Brownian motion through the level $\sqrt{2t}$.
\end{os}

In what follows we will consider the Mellin convolution $e^{\star n}_{\bar{\mu}}(x,t) = e_{\mu_1} \star \ldots \star e_{\mu_n}(x,t)$ (see formulae \eqref{MconvFormula} and \eqref{melConv2}) where $\bar{\mu} = (\mu_1, \ldots , \mu_n)$, $\mu_j>0$, $j=1,2,\ldots , n$ and, for the sake of simplicity, $e_{\mu}(x,t)=e^{1}_{\mu}(x, t )$. For the density law $e^{\star n}_{\bar{\mu}}(x,t)$, $x>0$, $t>0$ we have
\begin{equation}
\mathcal{M} \left[e^{\star n}_{\bar{\mu}}(\cdot, t) \right](\eta) = \prod_{j=1}^n  \mathcal{M} \left[e_{\mu_j}(\cdot, t^{1/n}) \right](\eta) = t^{\eta -1} \prod_{j=1}^n  \frac{\Gamma\left( \mu_j +1 -\eta \right)}{\Gamma\left( \mu_j \right)}
\label{mceW}
\end{equation}
with $\Re\{ \eta \} < 1$. Furthermore, for the Mellin convolution $g^{\gamma, \star n}_{\bar{\mu}}(x,t) = g^{\gamma}_{\mu_1} \star, \ldots , \star g^{\gamma}_{\mu_n}(x,t)$ we have
\begin{equation}
\mathcal{M}\left[ g^{\gamma, \star n}_{\bar{\mu}}(\cdot ,t) \right](\eta) = \prod_{j=1}^n  \mathcal{M}\left[ g^{\gamma}_{\mu_j}(\cdot ,t^{1/n}) \right](\eta) = t^{\eta -1} \prod_{j=1}^n \frac{\Gamma\left( \frac{\eta -1}{\gamma} + \mu_j \right)}{\Gamma\left( \mu_j \right)} 
\label{mcgW}
\end{equation}
with $\Re \{ \eta \} > 1 - \min_{j} \{ \mu_j \}$.
\begin{lm}
The functions $g^{\gamma}_{\mu}$ and $e^{\gamma}_{\mu}$ are commutative under $\star$-convolution. 
\label{LemmaComm}
\end{lm}
\begin{proof}
Consider the Mellin convolution \eqref{mceW}. Let $e_{\mu_j}$ be the distribution of the process $X^{\sigma_j}$, then formula \eqref{mceW} means that
\begin{equation*}
E\left\lbrace X^{\sigma_1}(X^{\sigma_2}(\ldots X^{\sigma_n}(t) \ldots)) \right\rbrace^{\eta -1} = E\left\lbrace X^{\sigma_1}(t^{1/n}) X^{\sigma_2}(t^{1/n}) \cdot \cdot \cdot X^{\sigma_n}(t^{1/n}) \right\rbrace^{\eta -1}   
\end{equation*}
for all possible permutations of $\{\sigma_j \}$, $j=1,2,\ldots , n$. The same result can be shown for $e^{\gamma}_{\mu_j}$. Suppose now that the process $X^{\sigma_j}$ possesses distribution $g^{\gamma}_{\mu_j}$, from \eqref{mcgW} we obtain the claimed result.
\end{proof}

\section{Stable subordinators}
\label{secSSub}
The $\nu$-stable subordinators $\{\tilde{\tau}^{(\nu)}_t$, $t>0\}$, $\nu \in (0,1)$, are defined as non-decreasing, (totally) positively skewed, L\'evy processes with  Laplace transform
\begin{equation} 
E \exp \{ - \lambda \tilde{\tau}^{(\nu)}_t \} = \exp\left\lbrace - t \lambda^{\nu} \right\rbrace , \quad t > 0, \quad \lambda >0
\label{transf:LapQ}
\end{equation}
and characteristic function
\begin{align}
E \exp\{ i \xi \tilde{\tau}^{(\nu)}_t \}  = &  \exp \{- t \Psi_{\nu}(\xi) \}, \quad \xi \in \mathbb{R}
\label{charFunSub}
\end{align}
where
\begin{equation*}
\Psi_{\nu} (\xi) = \int_0^\infty (1-e^{-i\xi u}) \frac{\nu}{\Gamma\left(1- \nu \right)} \frac{du}{u^{\nu +1}}
\end{equation*}
(see \citet{Btoi96, Zol86}). After some algebra we get
\begin{align*}
\Psi_{\nu} (\xi) = & \sigma |\xi |^{\nu} \left( 1 - i \, \textrm{sgn}(\xi) \tan\left( \frac{\pi \nu}{2} \right) \right) = |\xi |^{\nu} \exp\left\lbrace  -i \, \frac{\pi \nu}{2} \frac{\xi}{|\xi |}\right\rbrace .
\end{align*}
For the density law of the $\nu$-stable subordinator $\{ \tilde{\tau}^{(\nu)}_t,\, t>0\}$, say $h_{\nu}=h_\nu(x,t)$, $x>0$, $t>0$ we have the $t$-Mellin transforms 
\begin{equation}
\mathcal{M} \left[ \hat{h}_{\nu}(\xi, \cdot) \right](\eta) = |\xi |^{-\eta \nu} \exp\left\lbrace i \frac{\pi \eta \nu}{2} \frac{\xi}{|\xi |} \right\rbrace \Gamma\left( \eta \right) \label{MFt}
\end{equation}
and
\begin{equation}
\mathcal{M}\left[ \tilde{h}_{\nu}(\lambda, \cdot) \right](\eta) = \lambda^{-\eta \nu} \Gamma\left( \eta \right) \label{MLt}
\end{equation}
where $\hat{h}_\nu (\xi, t)= \mathcal{F}\left[h_{\nu}(\cdot, t) \right](\xi)$ is the Fourier transform appearing in \eqref{charFunSub} and $\tilde{h}_{\nu}(\lambda, t)=\mathcal{L}\left[ h_{\nu}(\cdot, t)\right](\lambda)$ is the Laplace transform \eqref{transf:LapQ}. By inverting \eqref{MFt} we obtain the Mellin transform with respect to $t$ of the density $h_{\nu}$ which reads
\begin{align}
\mathcal{M}\left[h_{\nu}(x, \cdot) \right] (\eta) =& \frac{1}{2\pi}\int_{\mathbb{R}} e^{-i\xi x} \mathcal{M} \left[ \hat{h}(\xi, \cdot) \right](\eta) d\xi \label{tmSub}\\
= & \frac{\Gamma\left(\eta \right) \Gamma\left( 1-\eta \nu \right)}{2 \pi} \left\lbrace \frac{e^{i\frac{\pi \eta \nu}{2}}}{(ix)^{1-\eta \nu}} + \frac{e^{-i \frac{\pi \eta \nu}{2}}}{(-i x)^{1-\eta \nu}}  \right\rbrace \nonumber \\
= & \frac{\Gamma\left(\eta \right) \Gamma\left( 1-\eta \nu \right)}{2 \pi \, x^{1-\eta \nu}} \left\lbrace \exp\left\lbrace  -i \frac{\pi}{2} + i \pi \eta \nu \right\rbrace  + \exp\left\lbrace i \frac{\pi}{2} - i \pi \eta \nu \right\rbrace  \right\rbrace \nonumber \\
= & \frac{\Gamma\left(\eta \right) \Gamma\left( 1-\eta \nu \right)}{\pi \, x^{1-\eta \nu}} \sin \pi \eta \nu =  \frac{\Gamma\left(\eta \right)}{\Gamma\left( \eta \nu \right)} x^{\eta \nu -1}, \quad x>0, \; \nu \in (0,1)\nonumber
\end{align}
where $\Re\{\eta \nu\} \in  (0,1)$. Formula \eqref{tmSub} can be also obtained by inverting \eqref{MLt}. We are also able to evaluate the Mellin transform with respect to $x$ of the density law $h_{\nu}$. From \eqref{MFt} and the fact that
\begin{equation}
\int_0^\infty x^{\eta -1} e^{-i \xi x} dx =  \frac{\Gamma(\eta)}{(i\xi)^\eta}, \quad \textrm{where} \quad (\pm i \xi)^\nu= |\xi |^\nu \exp\left\lbrace \pm i \frac{\nu \pi}{2} \frac{\xi}{|\xi |} \right\rbrace , \quad \nu \in (0,1) 
\label{mePPO}
\end{equation}
we obtain
\begin{align}
\mathcal{M}\left[ h_{\nu}(\cdot, t) \right](\eta) = & \frac{\Gamma\left( \eta \right)}{2\pi} \int_\mathbb{R} | \xi |^{-\eta} \exp\left\lbrace -i \frac{\pi \eta}{2} \frac{\xi}{|\xi |} - t \Psi_{\nu}(\xi)  \right\rbrace d\xi \nonumber \\
= & \frac{\Gamma(\eta)}{2\pi} \left\lbrace e^{-i \frac{\pi \eta}{2}} \int_0^\infty \xi^{-\eta} e^{-t \Phi_{\nu}(\xi)} d\xi + e^{i \frac{\pi \eta}{2}} \int_0^\infty \xi^{-\eta} e^{-t \Phi_{\nu}(-\xi)} d\xi \right\rbrace \nonumber \\
= & \frac{\Gamma(\eta)}{2\pi \nu} \Gamma\left( \frac{1-\eta}{\nu} \right) t^\frac{\eta -1}{\nu} \left\lbrace e^{i \pi (1-\eta)} + e^{-i \pi (1 -\eta)} \right\rbrace  \nonumber \\
= & \Gamma\left( \frac{1-\eta}{\nu} \right) \frac{t^\frac{\eta -1}{\nu}}{\nu \,\Gamma\left(1-\eta \right)}, \quad \Re\{ \eta \} \in (0,1), \, t>0. \label{melSsub}
\end{align}
The inversion of Fourier and Laplace transforms by making use of Mellin transform has been also treated by \citet{SWyss89}. \\

We investigate the relationship between stable subordinators and their inverse processes. For a $\nu$-stable subordinator $\{ \tilde{\tau}_t^{(\nu)}, \, t>0\}$ and an inverse process  $\{ L_t^{(\nu)}, \, t>0 \}$ (ISP in short) such that
\begin{equation*} 
Pr \{ L_t^{(\nu)} < x \} = Pr \{ \tilde{\tau}_x^{(\nu)} > t \}
\end{equation*}
we have the following relationship between density laws
\begin{equation}
l_{\nu}(x,t) =  Pr \{ L_t^{(\nu)} \in dx \} / dx = \int_t^\infty \frac{\partial}{\partial x} h_{\nu}(s, x)ds, \qquad x>0,\, t>0. \label{dist:lnu}
\end{equation}
We observe that $\frac{\partial}{\partial x} h_{\nu}(s, x)$ exists and there exists $\zeta(s) \in L^{1}(\mathbb{R}_{+})$ such that $\Xi_2(s) = \frac{\partial}{\partial x} h_{\nu}(s, x) = const\, \cdot D^{\nu}_s h_{\nu}(s,x) \leq \zeta(s)$. The function $h_{\nu}$ is the distribution of a totally skewed stable process, thus $h_{\nu}(\mathbf{x})$, $\mathbf{x} \in \mathbb{R}^{n}_{+}$ belongs to the space of functions in  $D((-\triangle)^{\nu/2})$, see \citet{SKM93}. Thus, the integral in \eqref{dist:lnu} converges. The density law \eqref{dist:lnu} can be written in terms of Fox functions by observing that
\begin{align}
\mathcal{M}\left[ l_{\nu}(\cdot, t) \right] (\eta) = & \int_t^\infty \mathcal{M}\left[\frac{\partial}{\partial x} h_{\nu}(s, \cdot) \right](\eta)\, ds  \nonumber \\
= & \left[ \textrm{by } \eqref{derMint} \right] = - (\eta -1) \int_t^\infty \mathcal{M}\left[ h_{\nu}(s, \cdot) \right](\eta -1) \, ds  \nonumber \\
= & \left[ \textrm{by } \eqref{tmSub} \right] = - \int_t^\infty \frac{\Gamma\left(\eta \right)}{\Gamma\left( \eta \nu -\nu \right)} s^{\eta \nu - \nu -1} ds\nonumber \\
= & \frac{\Gamma\left(\eta \right)}{\Gamma\left( \eta \nu - \nu +1\right)} t^{\nu (\eta-1)}, \quad \Re\{ \eta \} < 1/\nu,\, t>0. \label{xMSub}
\end{align}
Thus, by direct inspection of \eqref{mellinHfox}, we recognize that
\begin{equation}
l_{\nu}(x,t) = \frac{1}{t^\nu}H^{1,0}_{1,1}\left[ \frac{x}{t^\nu} \Bigg| \begin{array}{c} (1-\nu, \nu)\\ (0,1) \end{array} \right], \qquad x>0,\, t>0, \; \nu (0,1).
\label{dist:lH}
\end{equation}
Density \eqref{dist:lH} integrates to unity, indeed $\mathcal{M}\left[ l_{\nu}(\cdot, t) \right] (1) = 1$. The $t$-Laplace transform 
\begin{equation} 
\mathcal{L}[l_{\nu}(x, \cdot)](\lambda) = \lambda^{\nu-1} \exp\left\lbrace -x \lambda^\nu\right\rbrace , \qquad \lambda >0, \, \nu \in (0,1) \label{tranfLl}
\end{equation}
comes directly from the fact that
\begin{equation*}
\int_0^\infty e^{-\lambda t} \mathcal{M}\left[l_{\nu}(\cdot, t) \right](\eta) \, dt = \frac{\Gamma\left( \eta \right)}{\lambda^{\eta \nu -\nu +1}} =  \int_0^\infty x^{\eta -1} \mathcal{L}\left[ l_{\nu}(x, \cdot) \right](\lambda) \, dx .
\end{equation*}
From \eqref{tranfLl} we retrieve the well-known fact that $\mathcal{L}\left[ l_{\nu}(\cdot , t) \right](\lambda) = E_{\nu}(-\lambda t^{\nu})$ (see also \citet{BKS96}) where $E_{\beta}$ is the Mittag-Leffler function which can be also written as 
\begin{equation}
E_{\nu}(-\lambda t^{\nu}) = \frac{1}{\pi} \int_0^\infty \exp\left\lbrace -\lambda^{1/\nu}t x \right\rbrace  \frac{x^{\nu -1} \sin \pi \nu}{1 + 2x^{\nu} \cos \pi \nu + x^{2\nu}} dx, \quad t>0,\, \lambda >0. \label{E}
\end{equation}
The distribution $l_{\nu}$ satisfies the fractional equation $\frac{\partial^\nu}{\partial t^\nu} l_{\nu} = - \frac{\partial}{\partial x} l_{\nu}$, $x>0$, $t>0$ subject to $l_{\nu}(x,0)=\delta(x)$ where the fractional derivative must be understood in the Dzerbayshan-Caputo sense (formula \eqref{Cfracder}). The governing equation of $l_\nu$ can be also presented by considering the Riemann-Liouville derivative \eqref{Rfracder} and the relation \eqref{RCfracder} (see e.g. \citet{BM01, MS08, BMN09}). It is well-known that the ratio involving two independent stable subordinator $\{\,_1\tilde{\tau}^{(\nu)}_t, \, t>0\}$ and $\{\, _2\tilde{\tau}^{(\nu)}_t,\, t>0 \}$ has a distribution, $\forall t$, given by
\begin{equation}
r(w)=Pr\{ \, _1\tilde{\tau}^{(\nu)}_t /\, _2\tilde{\tau}^{(\nu)}_t \in dw \} /dw = \frac{1}{\pi} \frac{w^{\nu -1} \sin \pi \nu}{1+ 2 w^{\nu}\cos \pi \nu + w^{2\nu}}, \quad w>0, \, t>0.\label{ratioU}
\end{equation}
Here we study the ratio of two independent inverse stable processes $\{ \, _1L^{(\nu)}_t, \, t>0\}$ and $\{ \, _2L^{(\nu)}_t, \, t>0\}$  by evaluating its Mellin transform as follows
\begin{align}
E\left\lbrace \, _1L^{(\nu)}_t / \, _2L^{(\nu)}_t \right\rbrace^{\eta -1} = & \mathcal{M}\left[ l_{\nu}(\cdot, t) \right](\eta) \times \mathcal{M}\left[ l_{\nu}(\cdot, t) \right](2-\eta) = \frac{1}{\nu} \frac{\sin \nu \pi - \eta \nu \pi }{\sin \eta \pi } \label{masd}
\end{align}
with $\Re \{ \eta \} \in (0,1)$. By inverting \eqref{masd} we obtain
\begin{equation}
k(x)=\frac{1}{\nu \pi} \frac{\sin \nu \pi}{1 + 2 x \cos \nu \pi + x^2} = \frac{1}{2\pi i}\int_{\theta -i \infty}^{\theta + i \infty} \frac{\sin \nu \pi - \eta \nu \pi }{\sin \eta \pi } x^{-\eta} d\eta \label{ratioD}
\end{equation}
for some real $\theta \in (0,1)$. From \eqref{ratioU} and \eqref{ratioD} we can argue that
\begin{equation}
\left( _1\tilde{\tau}^{(\nu)}_t /\, _2\tilde{\tau}^{(\nu)}_t \right)^{\nu} \stackrel{law}{=} \, _1L^{(\nu)}_t / \, _2L^{(\nu)}_t, \quad \forall t>0. \label{indT}
\end{equation}
We notice that the equivalence in law \eqref{indT} is independent of $t$ as the formulae \eqref{ratioU} and \eqref{ratioD} entail. The distribution $h_{\nu} \circ l_{\nu}(x,t)$ of the process $\{\tilde{\tau}^{(\nu)}_{L^{(\nu)}_t}, \, t>0\}$ has Mellin transform  (by making use of the formulae \eqref{melSsub} and \eqref{xMSub}) given by
\begin{align*}
\mathcal{M}\left[ h_{\nu} \circ l_{\nu}(\cdot ,t) \right](\eta) = & \mathcal{M}\left[ h_{\nu}(\cdot ,1) \right](\eta) \times \mathcal{M}\left[ l_{\nu}(\cdot ,t) \right]\left(\frac{\eta -1}{\nu} +1 \right) = \frac{1}{\nu} \frac{\sin \pi \eta}{\sin \pi \frac{1-\eta}{\nu}}t^{\eta -1}, \quad  t>0
\end{align*}
with $\Re\{ \eta \} \in (0,1)$. Thus, we can infer that
\begin{equation*}
\tilde{\tau}^{(\nu)}_{L^{(\nu)}_t} \stackrel{law}{=} t \times \, _1\tilde{\tau}^{(\nu)}_t /\, _2\tilde{\tau}^{(\nu)}_t \quad t>0
\end{equation*}
and $h_{\nu} \circ l_{\nu}(x,t) = t^{-1}r(x/t)$ where $r(w)$ is that in \eqref{ratioU}. For the process $\{L^{(\nu)}_{\tilde{\tau}^{(\nu)}_t}, \; t>0 \}$ with distribution $l_{\nu} \circ h_{\nu}(x,t)$ we obtain (from \eqref{xMSub} and \eqref{melSsub})
\begin{align*}
\mathcal{M}\left[ l_{\nu} \circ h_{\nu}(\cdot, t) \right](\eta) = & \mathcal{M}\left[ l_{\nu}(\cdot, 1) \right](\eta) \times \mathcal{M}\left[ h_{\nu}(\cdot, t) \right](\eta \nu - \nu +1) = \frac{1}{\nu} \frac{\sin \pi \nu - \pi \eta \nu}{\sin \pi \eta} t^{\eta -1}
\end{align*}
with $\Re\{ \eta \} \in (0,1)$ and thus
\begin{equation*}
L^{(\nu)}_{\tilde{\tau}^{(\nu)}_t} \stackrel{law}{=} t \times \, _1L^{(\nu)}_t / \, _2L^{(\nu)}_t, \quad t>0.
\end{equation*}
We have that $l_{\nu} \circ h_{\nu}(x,t) = t^{-1} k(x/t)$ where $k(x)$ is that in \eqref{ratioD}.

\section{Main results}
\label{mainRes}
In this section we consider compositions of processes whose governing equations are (generalized) fractional diffusion equations. When we consider compositions involving Markov processes and stable subordinators we still have Markov processes.  Here we study Markov processes with random time which is the inverse of a stable subordinator. Such a process does not belong to the family of stable subordinators (see \eqref{E}) and the resultant composition is not, in general, a Markov process. This somehow explains the effect of the fractional derivative appearing in the governing equation, see \citet{MPG07}. Hereafter, we exploit the Mellin convolution of generalized Gamma densities in order to write explicitly the solutions to fractional diffusion equations. We first present a new representation of the density law $h_\nu$ by means of the convolution $e^{\star n}_{\bar{\mu}}$ introduced in Section \ref{secConv}. To do this we also introduce the time-stretching function $\varphi_m(s) = (s/m)^m$, $m \geq 1$, $s \in (0,\infty)$.
\begin{lm}
The Mellin convolution $e^{\star n}_{\bar{\mu}}(x, \varphi_{n+1}(t))$ where $\mu_j = j \, \nu$, for $j=1,2,\ldots, n$ is the density law of a $\nu$-stable subordinator $\{ \tilde{\tau}^{(\nu)}_t,\, t>0 \}$ with $\nu=1/(n+1)$, $n \in \mathbb{N}$. Thus, we have
\begin{equation*}
h_{\nu}(x,t) = e^{\star n}_{\bar{\mu}}(x, \varphi_{n+1}(t)), \quad x>0,\; t>0,\; \nu =1/(n+1), \; n \in \mathbb{N}.
\end{equation*}
\label{lemmaU}
\end{lm}
\begin{proof}
From \eqref{mceW} we have that
\begin{equation}
\mathcal{M}\left[ e^{\star n}_{\bar{\mu}}(\cdot, \varphi_{n+1}(t))  \right](\eta) = \frac{\prod_{j=1}^n \Gamma\left( 1-\eta +\mu_j \right)} {\prod_{j=1}^n \Gamma\left( \mu_j \right) } \left( \varphi_{n+1}(t) \right)^{\eta -1}. \label{tmpMq}
\end{equation}
From \citet[formula 8.335.3]{GR} we deduce that
\begin{equation}
\prod_{k=1}^{n} \Gamma\left( \frac{k}{n+1}\right) = \frac{(2\pi)^\frac{n}{2}}{\sqrt{n + 1}}, \quad n \in \mathbb{N}
\label{propGamma2}
\end{equation}
and formula \eqref{tmpMq} reduces to
\begin{equation}
\mathcal{M}\left[ e^{\star n}_{\bar{\mu}}(\cdot, \varphi_{n+1}(t))  \right](\eta) = \frac{\prod_{j=1}^n \Gamma\left( 1-\eta +\mu_j \right)} {(2\pi)^{n/2} \sqrt{\nu} } \left( \varphi_{n+1}(t) \right)^{\eta -1}. \label{tmpMq2}
\end{equation}
Furthermore, by making use of the (product theorem) relation
\begin{equation}
\Gamma(nx) = (2\pi)^\frac{1-n}{2} n^{nx -1/2} \prod_{k=0}^{n-1} \Gamma\left( x + \frac{k}{n} \right)
\end{equation}
(see \citet[formula 3.335]{GR}) formula \eqref{tmpMq2} becomes
\begin{align*}
\mathcal{M}\left[ e^{\star n}_{\bar{\mu}}(\cdot, \varphi_n(t))  \right](\eta) = & \frac{\Gamma\left( \frac{1-\eta}{\nu} \right) (2\pi)^{n /2} (n+1)^{\eta/\nu - n}}{\Gamma\left( 1- \eta \right) (2\pi)^{n/2}} \left( \varphi_{n+1}(t) \right)^{\eta -1} = \frac{\Gamma\left( \frac{1-\eta}{\nu} \right)}{\nu \, \Gamma\left( 1- \eta \right)} t^\frac{\eta -1}{\nu}
\end{align*}
(with $\Re\{ \eta \} \in (0,1)$) which coincides with \eqref{melSsub}. The claimed result is obtained.
\end{proof}
In light of the last result we are able to write explicitly the density law of a stable subordinator. For $\nu=1/2$, Lemma  \ref{lemmaU} says that
\begin{equation}
h_{1/2}(x,t) = e^{\star 1}_{\bar{\mu}} (x, \varphi_2(t)) = e_{1/2}(x, (t/2)^2) = \frac{x^{-1/2-1} e^{-\frac{t^2}{4x}}}{ t^{-1} \sqrt{4}\, \Gamma\left( \frac{1}{2} \right)}, \quad x>0, \, t>0
\end{equation}
which is the well-known density law of the $1/2$-stable subordinator or the first-passage time of a standard Brownian motion trough the level $t/\sqrt{2}$. For $\nu=1/3$, from \eqref{eConv}, we obtain 
\begin{equation*}
h_{1/3}(x,t)=e^{\star 2}_{\bar{\mu}}(x, \varphi_3(t)) = e_{1/3} \star e_{2/3} (x, (t/3)^3) =  \frac{1}{3 \pi} \frac{t^{3/2}}{x^{3/2}} K_\frac{1}{3}\left(\frac{2}{3^{3/2}} \frac{t^{3/2}}{\sqrt{x}} \right), \quad x>0,\, t>0.
\end{equation*}
For $\nu=1/4$, by \eqref{eConv} (and the commutativity under $\star$, see Lemma \ref{LemmaComm}), we have
\begin{equation*}
h_{1/4}(x,t)=e^{\star 3}_{\bar{\mu}}(x, \varphi_4(t)) =  e_{1/4} \star e_{2/4} \star e_{3/4} (x, (t/4)^4) = e_{1/2} \star (e_{1/4} \star e_{3/4}) (x, (t/4)^4) 
\end{equation*}
where $K_{1/2}(z)=\sqrt{\pi/2z}\exp\{-z\}$ (see \cite[formula 8.469]{GR}). We notice that
\begin{equation*}
\mathcal{M}\left[ e^{\star 3}_{\bar{\mu}}(\cdot, \varphi_4(t)) \right](\eta) = \mathcal{M}\left[ h_{1/2} \circ h_{1/2} (\cdot ,t) \right](\eta)
\end{equation*}
which is in line with the well-known fact that 
\begin{equation*}
E \exp\left\lbrace - \lambda \, _1\tilde{\tau}^{(\nu_1)}_{\, _2\tilde{\tau}^{(\nu_2)}_t}\right\rbrace = E \exp\left\lbrace - \lambda^{\nu_1} \, _2\tilde{\tau}^{(\nu_2)}_t \right\rbrace = \exp\{ - t \lambda^{\nu_1 \nu_2} \},  
\end{equation*}
$0 < \nu_i <1, \, i=1,2$.  For $\nu=1/5$, by exploiting twice \eqref{eConv} (and the commutativity under $\star$), we can write
\begin{align}
h_{1/5}(x,t)=e^{\star 4}_{\bar{\mu}}(x, (t/5)^5) = & (e_{1/5} \star e_{2/5}) \star (e_{3/5} \star e_{4/5}) (x, (t/5)^5) \label{dist:e4Unm} \\
= & \frac{t^{7/2}}{5^3 \pi^2 \, x^{3/10 +1}} \int_{0}^\infty s^{-2/5 -1} K_\frac{1}{5}\left( 2 \sqrt{\frac{s}{x}} \right) K_\frac{1}{5}\left( \frac{2}{5^{5/2}} \frac{t^{5/2}}{\sqrt{s}} \right) ds\nonumber
\end{align}
or equivalently 
\begin{align}
h_{1/5}(x,t)=e^{\star 4}_{\bar{\mu}}(x, (t/5)^5) = & (e_{1/5} \star e_{3/5}) \star (e_{2/5} \star e_{4/5}) (x, (t/5)^5) \label{dist:e4Dnm}\\
= & \frac{t^3}{5^{5/2} \pi^2 \, x^{2/5 +1}} \int_{0}^\infty s^{-1/5-1} K_\frac{2}{5}\left( 2 \sqrt{\frac{s}{x}} \right) K_\frac{2}{5}\left( \frac{2}{5^{5/2}} \frac{t^{5/2}}{\sqrt{s}} \right) ds. \nonumber
\end{align}
For $\nu=1/(2n+1)$, $n \in \mathbb{N}$, by using repeatedly \eqref{eConv} we arrive at
\begin{equation*}
h_{\nu}(x,t) = \frac{x^{\nu /2} \, t^{1/\nu- 3/2}}{\nu^{2-1/\nu} \pi^{1/2\nu - 1/2}} \mathcal{K}^{\circ\, n}_{\nu} \left( x, (\nu t)^{1/\nu} \right), \quad x>0,\, t>0
\end{equation*}
where
\begin{equation*}
\mathcal{K}^{\circ \,n}_{\nu}(x,t) = \int_0^\infty  \ldots \int_0^\infty  \mathcal{K}_{\nu}(x,s_1) \ldots \mathcal{K}_{\nu}(s_{n-1},t)\, ds_1 \ldots ds_{n-1}
\end{equation*}
is the integral \eqref{circConv} (as the symbol "$\circ \, n$" denote) where $n$ functions are involved and $\mathcal{K}_{\nu}(x,t) = x^{-2\nu -1} K_{ \nu}\left( 2 \sqrt{t/x} \right)$, $x>0$,  $t>0$. We state a similar result for the density law $l_\nu$ and the convolution $g^{\gamma, \star n}_{\bar{\mu}}$ (see Section \ref{secConv}). Let us consider the time-stretching function $\psi_m(s) = m \, s^{1/m}$, $s \in (0, \infty)$, $m \in \mathbb{N}$, ($\psi = \varphi^{-1}$ where $\varphi$ has been introduced in the previous Lemma). 
\begin{lm}
The Mellin convolution $g^{(n+1), \star n}_{\bar{\mu}}(x, \psi_{n+1}(t))$ where $\mu_j  = j\, \nu$, $j=1,2,\ldots ,  n$ and $\nu=1/(n+1)$, $n \in \mathbb{N}$, is the density law of a $\nu$-inverse process $\{L^{(\nu)}_t, \, t>0 \}$. Thus, we have
\begin{equation*}
l_{\nu}(x,t) = g^{(n+1), \star n}_{\bar{\mu}}(x, \psi_{n+1}(t)), \quad x>0, \; t>0, \; \nu=1/(n+1),\; n \in \mathbb{N}.
\end{equation*}
\label{LemmaD}
\end{lm}
\begin{proof}
The proof can be carried out as the proof of Lemma \ref{lemmaU}.
\end{proof}
We obtain that $ l_{1/2}(x,t) = g^2_{1/2}(x, 2t^{1/2})  = e^{-\frac{x^2}{4t}}/\sqrt{\pi t}$, $x>0$, $t>0$. Moreover, by making use of \eqref{eConv} and \eqref{propGamma2}, we have that
\begin{equation*}
l_{1/3}(x,t) = g^3_{2/3} \star g^3_{1/3}(x, 3t^{1/3}) =  \frac{1}{\pi} \sqrt{\frac{x}{t}} K_\frac{1}{3}\left( \frac{2}{3^{3/2}}\frac{x^{3/2}}{\sqrt{t}} \right), \quad x>0,\; t>0
\end{equation*}
and $l_{1/4}(x,t) =g^4_{3/4} \star g^4_{2/4} \star g^4_{1/4}(x, 4t^{1/4})$ follows (thank to the commutativity under $\star$) from
\begin{align*}
g^4_{3/4} \star g^4_{2/4} \star g^4_{1/4}(x, t) = & g^4_{1/2} \star (g^4_{3/4} \star g^4_{1/4})(x, t) =  \frac{2^{3/2}}{\pi} \frac{x}{t} \int_0^\infty \exp\left\lbrace -(sx)^4 - \frac{2}{(st)^2}\right\rbrace ds
\end{align*}
where $g^4_{3/4} \star g^4_{1/4}(x,t)$ is given by \eqref{eConv} and $K_{1/2}(z)=\sqrt{\pi/2z}\exp\{-z\}$ (see \cite[formula 8.469]{GR}).  In a more general setting, by making use of \eqref{eConv} we can write down
\begin{equation}
g^{1/\nu, \star (1 /\nu - 1)}_{\bar{\mu}}(x, t) = \frac{1}{\nu^{1/2\nu}} \left( \frac{x}{\pi^2 t^3} \right)^\frac{1-\nu}{4\nu} \mathcal{Q}^{\circ n}_{\frac{1-\nu}{2}}(x,t), \quad \nu =1/(2n+1), \; n \in \mathbb{N} \label{genl}
\end{equation}
where the symbol ''$\circ \, n$ '' stands for the integral \eqref{circConv} where $n$ functions $\mathcal{Q}_{\frac{1-\nu}{2}}$ are involved and 
\begin{equation}
\mathcal{Q}_{\frac{1-\nu}{2}}(x,t)=K_{\frac{1-\nu}{2}} \left( 2\sqrt{ (x / t)^{1/\nu}} \right), \quad x>0,\; t>0. \label{Qrond}
\end{equation}

Now, we present the main result of this paper concerning the explicit solutions to (generalized) fractional diffusion equations. We study a generalized problem which leads to fractional diffusion equations involving the adjoint operators of both Bessel and squared Bessel processes. Let us introduce the distribution $\tilde{u}^{\gamma, \mu}_{\nu} = \tilde{g}^{\gamma}_{\mu} \circ l_{\nu}$ where $\tilde{g}^\gamma_\mu(x,t) = g^\gamma_{\mu}(x, t^{1/\gamma})$ and the Mellin transform of $\tilde{u}^{\gamma, \mu}_{\nu}$ which reads
\begin{equation}
\mathcal{M}\left[ \tilde{u}^{\gamma, \mu}_{\nu}(\cdot, t) \right](\eta) = \frac{\Gamma\left( \frac{\eta -1}{\gamma} +\mu \right) \Gamma\left( \frac{\eta -1}{\gamma} +1 \right)}{\Gamma(\mu) \Gamma\left( \frac{\eta -1}{\gamma}\nu +1 \right)} t^{\frac{\eta -1}{\gamma}\nu}, \quad 1-\gamma \mu < \Re\{\eta \} < 1+ \gamma /\nu - \gamma.
\label{MelFracq}
\end{equation}
We state the following result.
\begin{te}
Let the previous setting prevail. For $\nu=1/(2n+1)$, $n \in \mathbb{N} \cup \{0\}$, the solutions to
\begin{equation}
D^{\nu}_t \, \tilde{u}^{\gamma, \mu}_{\nu} = \mathcal{G}_{\gamma,\mu} \, \tilde{u}^{\gamma, \mu}_{\nu}, \quad x > 0,\; t>0 \label{ProbFrac}
\end{equation}
can be represented in terms of generalized Gamma convolution as
\begin{equation}
\tilde{u}^{\gamma, \mu}_{\nu} (x,t) =  \gamma \frac{x^{2\mu -1} \nu^\frac{1-3\nu}{4\nu}}{(\pi^2 t^{3\nu})^\frac{1-\nu}{4\nu}} \int_0^\infty s^{\frac{1}{4\nu} - \frac{1}{4} -\mu} e^{- x^\gamma / s}  \, v_{\nu}(s,t) \, ds, \quad x \geq 0, \; t>0 \label{solFracq}
\end{equation}
where $\mathcal{G}_{\gamma,\mu}$ is the operator appearing in \eqref{generatorG},
\begin{equation*}
v_{\nu}(s,t) = \int_0^\infty \ldots \int_0^\infty \mathcal{Q}_{\frac{1-\nu}{2}}(s, s_1) \ldots \mathcal{Q}_{\frac{1-\nu}{2}}(s_{n-1}, t^{\nu}/\nu) ds_1 \ldots ds_{n-1}
\end{equation*}
and $\mathcal{Q}_{\frac{1-\nu}{2}}$ is that in \eqref{Qrond}. Moreover, for $\nu \in (0,1]$, we have
\begin{equation}
\tilde{u}^{\gamma, \mu}_{\nu} (x,t) = \frac{\gamma}{x t^{\nu / \gamma}} H^{2,0}_{2,2} \left[ \frac{x^{\gamma}}{t^{\nu}} \Bigg| \begin{array}{cc} (1,\nu);& (\mu, 0)\\ (1,1);&(\mu , 1)\end{array} \right]
\label{FoxDifa}
\end{equation}
in terms of H Fox functions.
\label{mainTheorem}
\end{te}
\begin{proof}
By exploiting the property \eqref{derMint} of the Mellin transform and the fact that
\begin{equation*}
\int_0^\infty x^{\eta -1} x^{\theta} f(x) dx = \mathcal{M}\left[ f(\cdot) \right](\eta +\theta),
\end{equation*}
for the operator \eqref{generatorG} we have that
\begin{align}
 \mathcal{M} & \left[ \mathcal{G}_{\gamma,\mu} \tilde{u}^{\gamma, \mu}_{\nu}(\cdot, t) \right] (\eta)\nonumber \\  
= & - \frac{1}{\gamma^2} (\eta -1) \mathcal{M}\left[ \frac{\partial}{\partial x} \tilde{u}^{\gamma, \mu}_{\nu}(\cdot, t) \right] (\eta -\gamma +1) + \frac{1}{\gamma^2}(\gamma \mu -1) (\eta -1) \mathcal{M}\left[\tilde{u}^{\gamma, \mu}_{\nu}(\cdot, t) \right](\eta - \gamma)\nonumber \\
= & \frac{1}{\gamma^2}(\eta -1) (\eta -\gamma) \mathcal{M}\left[ \tilde{u}^{\gamma, \mu}_{\nu}(\cdot, t) \right] (\eta -\gamma) + \frac{1}{\gamma^2} (\gamma \mu -1) (\eta -1) \mathcal{M}\left[\tilde{u}^{\gamma, \mu}_{\nu}(\cdot, t) \right](\eta - \gamma)\nonumber \\
= & \frac{1}{\gamma^2} (\eta -1) (\eta - 1 + \gamma \mu - \gamma ) \mathcal{M}\left[\tilde{u}^{\gamma, \mu}_{\nu}(\cdot, t) \right](\eta - \gamma) \label{compDer}
\end{align}
where $\mathcal{M}\left[\tilde{u}^{\gamma, \mu}_{\nu}(\cdot, t) \right](\eta)$ is that in \eqref{MelFracq}. We obtain
\begin{align*}
\mathcal{M} \left[ \mathcal{G}_{\gamma,\mu} \tilde{u}^{\gamma, \mu}_{\nu}(\cdot, t) \right](\eta) = & \frac{1}{\gamma^2} (\eta -1) (\eta - 1 + \gamma \mu - \gamma ) \frac{\Gamma\left( \frac{\eta -\gamma -1}{\gamma} +\mu \right) \Gamma\left( \frac{\eta -\gamma -1}{\gamma} +1 \right)}{\Gamma(\mu) \Gamma\left( \frac{\eta - \gamma -1}{\gamma}\nu +1 \right)} t^{\frac{\eta -\gamma -1}{\gamma}\nu}\\
= & \frac{1}{\gamma} (\eta -1) \frac{\Gamma\left( \frac{\eta -1}{\gamma} +\mu \right) \Gamma\left( \frac{\eta -1}{\gamma} \right)}{\Gamma(\mu) \Gamma\left( \frac{\eta -1}{\gamma}\nu - \nu +1 \right)} t^{\frac{\eta -1}{\gamma}\nu - \nu}\\
= & \frac{\Gamma\left( \frac{\eta -1}{\gamma} +\mu \right) \Gamma\left( \frac{\eta -1}{\gamma} + 1\right)}{\Gamma(\mu) \Gamma\left( \frac{\eta -1}{\gamma}\nu - \nu +1 \right)} t^{\frac{\eta -1}{\gamma}\nu - \nu} = D^{\nu}_t \mathcal{M} \left[ \tilde{u}^{\gamma, \mu}_{\nu}(\cdot, t) \right] (\eta)
\end{align*}
and $\tilde{u}^{\gamma, \mu}_{\nu}(x, t)$ solves \eqref{ProbFrac} for $\nu \in (0,1)$. In view of Lemma \ref{LemmaD} we can write
\begin{equation*}
\tilde{u}^{\gamma, \mu}_{\nu}(x, t) = \int_0^\infty \tilde{g}^\gamma_\mu (x,s) \; g^{1/\nu, \star (1 /\nu - 1)}_{\bar{\mu}}(s, \psi_{1/\nu}(t)) \; ds
\end{equation*}
and by means of \eqref{genl} result \eqref{solFracq} appears. Formula \eqref{FoxDifa} follows directly from \eqref{mellinHfox} by considering formula \eqref{propH1} and the fact that
\begin{equation}
\H = \frac{1}{x^c} \; H^{m,n}_{p,q}\left[ x \bigg| \begin{array}{l} (a_i + c \alpha_i, \alpha_i)_{i=1, .. , p}\\ (b_j + c \beta_j, \beta_j)_{j=1, .. , q}  \end{array} \right] \label{propH2}
\end{equation}
for all $c \in \mathbb{R}$ (see \citet{MS73}). 
\end{proof}

We specialize the previous result by keeping in mind formula \eqref{RCfracder} and the operator \eqref{generatorG}.
\begin{coro}
For $\nu=1$, Theorem \ref{mainTheorem} says that
\begin{equation*}
\frac{\partial}{\partial t} \tilde{u}^{\gamma, \mu}_{1} = \frac{1}{\gamma^2} \left( \frac{\partial}{\partial x} x^{2-\gamma} \frac{\partial}{\partial x} - (\gamma \mu -1) \frac{\partial}{\partial x} x^{1-\gamma} \right) \tilde{u}^{\gamma, \mu}_{1}, \quad x > 0, \; t>0,\; \gamma \neq 0
\end{equation*}
where $\tilde{u}^{\gamma, \mu}_{1} = \tilde{g}^{\gamma}_{\mu}$ is the distribution of the GGP.
\label{coro1}
\end{coro}
This is because $L^{1}_t \stackrel{a.s.}{=} t$. Indeed, for $\nu = 1$, $L^{\nu}_t$ is the elementary subordinator (see \cite{Btoi96}).
\begin{proof}
If $\nu=1$, then the equation \eqref{MelFracq} takes the form
\begin{equation}
\Psi_t(\eta) = \mathcal{M}[\tilde{g}^{\gamma}_{\mu}(\cdot , t)](\eta)= \Gamma\left( \frac{\eta -1}{\gamma} + \mu \right) \frac{t^\frac{\eta -1}{\gamma}}{\Gamma(\mu)}, \quad \Re\{ \eta \} > 1-\gamma \mu 
\label{Zyyy}
\end{equation}
where $\tilde{g}^{\gamma}_{\mu}(x, t) = g^{\gamma}_{\mu}(x, t^{1/\gamma})$. For $\gamma>0$, we perform the time derivative of \eqref{Zyyy} and obtain
\begin{align*}
\frac{\partial}{\partial t} \Psi_t(\eta) = &\frac{\eta -1}{\gamma} \Gamma\left( \frac{\eta -1}{\gamma} + \mu \right) t^\frac{\eta -\gamma -1}{\gamma}\\
=  & \frac{\eta -1}{\gamma} \left( \frac{\eta - \gamma -1 + \gamma \mu}{\gamma} \right) \Gamma\left( \frac{\eta - \gamma -1}{\gamma} + \mu \right) t^\frac{\eta - \gamma -1}{\gamma}\\
= & \frac{1}{\gamma^2} (\eta -1) ( \eta - \gamma -1 + \gamma \mu ) \Psi_t(\eta - \gamma)
\end{align*}
which coincides with \eqref{compDer} and $\tilde{u}^{\gamma, \mu}_{1}(x,t) = \tilde{g}^{\gamma}_{\mu}(x, t)$, $\gamma>0$. Similar calculation must be done for $\gamma < 0$ and the proof is completed.
\end{proof}

\begin{coro}
Let us write $\tilde{u}^{\mu}_{\nu}(x,t)=\tilde{u}^{1, \mu}_{\nu}(x,t)$. The distribution $\tilde{u}^\mu_\nu (x,t)$, $x>0$, $t>0$ $\mu>0$, $\nu \in (0,1]$, solves the following fractional equation
\begin{equation}
\frac{\partial^{\nu}}{\partial t^{\nu}}u^{\mu}_{\nu} = \left( x \frac{\partial^2}{\partial x^2} - (\mu -2) \frac{\partial}{\partial x} \right) u^{\mu}_{\nu}. \label{pde:fracG1}
\end{equation}
\end{coro}

In particular, for $\nu = 1/2$, we have
\begin{equation*}
\tilde{u}^{\mu}_{1/2}(x,t) = \frac{x^{\mu -1}}{\sqrt{\pi t} \Gamma(\mu)} \int_0^\infty s^{-\mu} \exp\left\lbrace -\frac{x}{s} - \frac{s^2}{4t} \right\rbrace ds, \quad x>0,\, t>0,\; \mu>0
\end{equation*}
which can be seen as the distribution of the process $\{G^{1,\mu}_{|B(2t)|}$,\; $t>0\}$ where $B$ is a standard Brownian motion run at twice its usual speed and $G^{\gamma, \mu}_t$ is a GGP. We notice that the process $G^{1,\mu}_t$ is a squared Bessel process starting from zero.

\begin{coro}
The distribution $\tilde{u}^{2,\mu}_{\nu} = \tilde{u}^{2,\mu}_{\nu}(x,t)$, $x>0$, $t>0$, $\mu>0$, $\nu \in (0,1]$ solves the following fractional equation
\begin{equation*}
\frac{\partial^\nu}{\partial t^{\nu}} \tilde{u}^{2,\mu}_{\nu} = \frac{1}{2^2} \left( \frac{\partial^2}{\partial x^2} - \frac{\partial}{\partial x} \frac{(2\mu -1)}{x} \right) \tilde{u}^{2,\mu}_{\nu}.
\end{equation*}
\end{coro}

In particular, for $\nu=1/3$, we have
\begin{equation*}
\tilde{u}^{2, \mu}_{1/3}(x,t) = \frac{2\, x^{2\mu -1}}{\pi \Gamma(\mu) \sqrt{t}} \int_0^\infty \frac{e^{-\frac{x^2}{s}}}{s^{\mu -1/2}} K_\frac{1}{3}\left(\frac{2}{3^{3/2}} \frac{s^{3/2}}{\sqrt{t}} \right) ds, \quad x>0, \; t>0,\; \mu>0
\end{equation*}
and for $\mu=1/2$ we obtain 
\begin{equation*}
\tilde{u}^{2, 1/2}_{1/3}(x,t) = \frac{2}{\pi^{3/2} \sqrt{t}} \int_0^\infty e^{-\frac{x^2}{s}} K_\frac{1}{3}\left(\frac{2}{3^{3/2}} \frac{s^{3/2}}{\sqrt{t}} \right) ds, \quad x>0, \; t>0
\end{equation*}
which is the distribution of $|B(L^{1/3}_t)|$ where $|B(t)|$ is a folded Brownian motion with variance $t/2$.\\

\textbf{Acknowledgement} The author is grateful to the anonymous referee for careful checks and comments.

\bibliographystyle{abbrvnat}
\bibliography{biblio} 

\end{document}